\documentclass[10pt]{amsart}
\usepackage{enumerate,amsmath,amssymb,latexsym,
amsfonts, amsthm, amscd, MnSymbol}

\theoremstyle{plain}

\newtheorem{theorem}{Theorem}

\numberwithin{equation}{section}

\newcommand{\R}{\mathbb{R}}
\newcommand{\C}{\mathbb{C}}

\newcommand{\sn}{{\rm sn}}

\newcommand{\ii}{{\rm i}}
\newcommand{\dd}{{\rm dn}_2}

\newcommand{\sss}{{\rm sn}_2}

\newcommand{\dk}{{\rm dn}_{\kappa}}
\newcommand{\dl}{{\rm dn}_{\lambda}}
\newcommand{\Kk}{K_{\kappa}}
\newcommand{\Kl}{K_{\lambda}}

\begin{document}

\title {The elliptic function ${\rm dn}_2$ of Shen}

\date{}

\author[P.L. Robinson]{P.L. Robinson}

\address{Department of Mathematics \\ University of Florida \\ Gainesville FL 32611  USA }

\email[]{paulr@ufl.edu}

 \keywords{}

\begin{abstract}

We analyze the elliptic function ${\rm dn}_2$ introduced by Li-Chien Shen, contributing to the Ramanujan theory of elliptic functions in signature four. 

\end{abstract}

\maketitle

\medbreak

\section*{Introduction}

\medbreak 

In [4] Li-Chien Shen introduced and studied  a signature-four counterpart $\dd$ to the classical Jacobian elliptic function ${\rm dn}$. He showed that the function $\dd$ is elliptic, expressing it in terms of its coperiodic Weierstrass function and in terms of associated Jacobian functions; he also located its poles and gave expressions for its fundamental periods. The formulation throughout was based on theta functions, and included much more information than this summary suggests. 

\medbreak 

Our aim in the present paper is twofold. On the one hand, we amplify several aspects of the account presented in [4] and view matters from a different perspective, completely avoiding theta functions. On the other hand, we reinforce the position of $\dd$ in Ramanujan's theory of elliptic functions to alternative bases, in part by showing how $\dd$ enables the establishment of hypergeometric identities relating signature four both to itself and to the classical signature. 

\medbreak 

\section*{Preliminary results}

\medbreak 

Our primary purpose in this opening section is to construct the elliptic function $\dd$ and present some of its most basic properties. We do not aim at anything approaching completeness: for further detail, see the original account [4] in which $\dd$ made its first appearance; see also the partial reworking [3] in line with the present paper.  

\medbreak 

We begin by fixing a modulus $\kappa \in (0, 1)$ and defining 
$$f(T) = \int_0^T F(\tfrac{1}{4}, \tfrac{3}{4}; \tfrac{1}{2}; \kappa^2 \sin^2 t) \, {\rm d} t$$
where $F = \, _2F_1$ is the hypergeometric function with the indicated parameters. Here, $f$ is analytic and invertible in a sufficiently small open neighbourhood of $0$; its analytic inverse $\phi$ is then defined in a suitably small open disc $D$ about $0$ and an auxiliary analytic function $\psi$ is defined in $D$ by $\psi(0) = 0$ and the requirement $\sin \psi = \kappa \sin \phi$. 

\medbreak 

Now $\dd$ is initially the analytic function defined in $D$ by the rule 
$$\dd = \cos \psi.$$ 
A companion analytic function $\sss$ is defined in $D$ by the rule 
$$\sss = \sin \phi.$$ 
Note that in $D$ there holds the counterpart 
$$\dd^2 = 1 - \sin^2 \psi = 1 - \kappa^2 \sin^2 \phi = 1 - \kappa^2 \sss^2$$ 
to the identity ${\rm dn}^2 + \kappa^2 {\rm sn}^2 = 1$ familiar from the theory of Jacobian elliptic functions. 

\medbreak 

Here, the precise choice of $D$ is largely immaterial: it turns out that $\dd$ as defined above is in fact the restriction to $D$ of an elliptic function, this elliptic function being $\dd$ proper. 

\medbreak 

It is now convenient to introduce the complementary modulus $\lambda = \sqrt{1 - \kappa^2} \in (0, 1)$. 

\medbreak 

\begin{theorem} \label{DE}
The function $\dd$ satisfies $\dd(0) = 1$ and the differential equation 
$$(\dd ')^2 = 2 (1 - \dd) ( \dd^2 - \lambda^2).$$
\end{theorem} 

\begin{proof} 
As $\phi$ is inverse to $f$ and as 
$$f ' \circ \phi = F(\tfrac{1}{4}, \tfrac{3}{4}; \tfrac{1}{2}; \kappa^2 \sin^2 \phi) = F(\tfrac{1}{4}, \tfrac{3}{4}; \tfrac{1}{2}; \sin^2 \psi)$$ 
it follows from the standard hypergeometric evaluation 
$$F(\tfrac{1}{4}, \tfrac{3}{4}; \tfrac{1}{2}; \sin^2 \psi) = \frac{\cos \frac{1}{2} \psi}{\cos \psi}$$
that 
$$\phi ' = \frac{\cos \psi}{\cos \frac{1}{2} \psi}.$$
Additionally, $\sin \psi = \kappa \sin \phi$ implies 
$$\psi ' = \frac{\kappa \cos \phi}{\cos \psi}$$ 
while $\dd = \cos \psi$ implies 
$$\dd ' = - \sin \psi \, \psi '.$$ 
As $\sin \psi = 2 \sin \frac{1}{2} \psi \cos \frac{1}{2} \psi$ it follows that 
$$\dd ' = - 2 \, (\sin \frac{1}{2} \psi) (\kappa \cos \phi).$$ 
As $\cos \psi = 1 - 2 \sin^2 \frac{1}{2} \psi$ and $\kappa^2 - \sin^2 \psi = \cos^2 \psi - \lambda^2$ it follows that 
$$(\dd ')^2 = 2 (1 - \dd) ( \dd^2 - \lambda^2).$$
\end{proof} 

\medbreak 

At present, $\dd$ is defined locally around $0$. The following result allows us to globalize. 

\medbreak 

\begin{theorem} \label{dell}
The function $\dd$ satisfies 
$$(1 - \dd) (\tfrac{1}{3} + \wp) = \tfrac{1}{2} \kappa^2$$ 
where $\wp$ is the Weierstrass elliptic function with invariants 
$$g_2 = \tfrac{4}{3} - \kappa^2 = \lambda^2 + \tfrac{1}{3}$$ 
$$g_3 = \tfrac{8}{27} - \tfrac{1}{3} \kappa^2 = \tfrac{1}{3} \lambda^2 - \tfrac{1}{27}.$$ 
\end{theorem} 

\begin{proof} 
The function 
$$p = - \frac{1}{3} + \frac{\frac{1}{2} \kappa^2}{1 - \dd}$$
has a pole at $0$ and direct calculation shows that it satisfies 
$$(p ')^2 = 4 p^3 - (\lambda^2 + \tfrac{1}{3}) p - (\lambda^2 - \tfrac{1}{27}).$$ 
This initial value problem characterizes the Weierstrass function with the stated invariants. 
\end{proof} 

\medbreak 

Briefly, we may say that $\dd$ is the elliptic function given by 
$$\dd = 1 - \tfrac{1}{2} \kappa^2 / (\tfrac{1}{3} + \wp).$$ 
Were greater clarity desirable, we could initially define ${\rm d} = \cos \psi$ on $D$: from ${\rm d}(0) = 1$ and the differential equation $({\rm d} ')^2 = 2 (1 - {\rm d}) ({\rm d}^2 - \lambda^2)$ it would then follow that ${\rm d}$ is the restriction to $D$ of this elliptic function. 

\medbreak 

The invariants of $\wp$ being real and its discriminant 
$$\Delta = g_2^3 - 27 g_3^2 = \kappa^4 \lambda^2$$ 
being positive, the period lattice of $\wp$ is rectangular. Let $2 \omega = 2 K$ and $2 \omega ' = 2 \ii K '$ be fundamental periods with $K > 0$ and $K ' > 0$. It is a familiar fact that $\wp$ has purely real values on the edges of the half-period rectangle $Q$ with vertices 
$$0, \, K, \, K + \ii K ', \, \ii K ';$$ 
also that the values of $\wp$ decrease strictly from $+ \infty$ to $- \infty$ upon counterclockwise passage around the perimeter of $Q$ starting from and returning to $0$. The (midpoint) values of $\wp$ at the non-zero vertices of $Q$ are the zeros of the cubic $4 w^3 - g_2 w - g_3$ in $w$: explicitly, 
$$\wp(K) = \tfrac{1}{6} + \tfrac{1}{2} \lambda, \, \wp( K + \ii K ') = \tfrac{1}{6} - \tfrac{1}{2} \lambda, \, \wp( \ii K ') = - \tfrac{1}{3}.$$ 

\bigbreak 

\begin{theorem} \label{values}
The elliptic function $\dd$ has $2 K$ and $2 \ii K '$ as fundamental periods. It has a double pole at $\ii K '$; also, $\dd(K) = \lambda$ and $\dd(K + \ii K ') = - \lambda.$ 
\end{theorem} 

\begin{proof} 
Theorem \ref{dell} shows that the elliptic functions $\dd$ and $\wp$ are coperiodic. Moreover, the lattice midpoint $\ii K '$ is a double zero of $\frac{1}{3} + \wp$ and hence a double pole of $\dd$. The values of $\dd$ at $K$ and $K + \ii K '$ follow from those of $\wp$. 
\end{proof} 

\medbreak 

Here, note that the values of $\dd$ decrease strictly from $+ \infty$ to $- \infty$ upon counterclockwise passage around the perimeter of $Q$ starting from and returning to $\ii K '$. 

\medbreak 

The Weierstrass function $\wp$ has an associated triple ${\rm sn}, \, {\rm cn}, \, {\rm dn}$ of classical Jacobian elliptic functions. With the semistandard notation 
$$e_1 = \wp(K), \, e_2 = \wp(K + \ii K '), \, e_3 = \wp(\ii K ')$$
for the midpoint values in decreasing order, these Jacobian functions have modulus $k \in (0, 1)$ given by 
$$k^2 = \frac{e_2 - e_3}{e_1 - e_3}$$
and the Jacobian modular sine function ${\rm sn}$ is related to $\wp$ by 
$$\wp (z) = e_3 + \frac{e_1 - e_3}{{\rm sn}^2 [z \,\sqrt{e_1 - e_3} ]}.$$ 

\medbreak 

\begin{theorem} \label{sn}
The elliptic function $\dd$ satisfies 
$$\dd(z) = 1 - (1 - \lambda) {\rm sn}^2 \Big[z \, \sqrt{\frac{1 + \lambda}{2}}\Big]$$
where ${\rm sn} = {\rm sn} (\bullet, k)$ is the Jacobian sine function with modulus $k$ given by 
$$k^2 = \frac{1 - \lambda}{1 + \lambda}.$$
\end{theorem} 

\begin{proof} 
From Theorem \ref{dell}, noting that $e_1 - e_3 = (1 + \lambda) / 2$ and $e_2 - e_3 = (1 - \lambda) / 2$. 
\end{proof} 

\medbreak 

We close this section by returning to the initial definition of $\dd$. Note that the formula for $f(T)$ is valid whenever $T$ is real, thereby defining a function $f : \R \to \R$ whose periodic derivative is strictly positive. The inverse $\phi$ of $f$ is therefore defined on the whole real line: if $u \in \R$ then 
$$u = \int_0^{\phi (u)} F(\tfrac{1}{4}, \tfrac{3}{4}; \tfrac{1}{2}; \kappa^2 \sin^2 t) \, {\rm d} t.$$
More is true. From the proof of Theorem \ref{DE} we see that 
$$(\phi ')^2 = \frac{2 \cos^2 \psi}{1 + \cos \psi} = \frac{2 \, \dd^2}{1 + \dd}.$$
Here, $\dd$ has poles at $\pm \ii K '$, while taking the value $-1 \in (- \infty, - \lambda)$ at a point $w$ on the upper edge of the rectangle $Q$ and at the conjugate point $\overline{w}$. It follows (by `square-rooting') that $\phi '$ extends holomorphically to the open band $\{ z \in \C : |{\rm Im} \, z| < K ' \}$ and then (by integrating from $\phi(0) = 0$) that $\phi$ itself extends holomorphically to the same band. Similarly, as the zeros of $1 - \dd$ (at points congruent to $0$ modulo periods) are double, the identity 
$$\dd^2 + \kappa^2 \sss^2 = 1$$
implies that $\sss$ extends holomorphically to the very same band, throughout which the relation $\sss = \sin \circ \phi$ continues to hold. 

\medbreak 

\section*{Fundamental periods}

\medbreak 

Here, we derive explicit formulae for the fundamental periods $2 K$ and $2 \ii K'$ of the elliptic function $\dd$. In fact, we derive more than one formula for each of these periods and take more than one approach to do so. Our formulae express these periods in terms of the modulus $\kappa$ and complementary modulus $\lambda = \sqrt{1 - \kappa^2}$. 

\medbreak 

In the interests of clarity, we shall replace $\dd$ by $\dk$ as a symbol for the elliptic function hitherto constructed with $\kappa$ as modulus, making similar replacements when referring to fundamental periods and half-period rectangles. The complementary modulus $\lambda$ engenders a corresponding elliptic function $\dl$ with fundamental periods $2 \Kl$ and $2 \ii \Kl '$. 

\medbreak 

Recall from Theorem \ref{DE} that $\dk$ satisfies the differential equation 
$$(\dk ')^2 = 2 (1 - \dk) (\dk^2 - \lambda^2).$$
Recall also the behaviour of $\dk$ on the rectangle $Q_{\kappa}$: it decreases strictly from $\dk (0) = 1$ to $\dk (\Kk) = \lambda$ along the lower edge $[0, \Kk]$ and decreases strictly from $\dk (\Kk) = \lambda$ to $\dk (\Kk + \ii \Kk ') = - \lambda$ along the right edge $[\Kk, \Kk + \ii \Kk ']$. It follows that 
$$\dk ' = - \sqrt{2 (1 - \dk) (\dk^2 - \lambda^2)}$$ 
along $[0, \Kk]$ and therefore that 
$$\int_0^{\Kk} 1 = - \int_{\dk(0)}^{\dk(\Kk)}\frac{{\rm d} x}{\sqrt{2 (1 - x) (x^2 - \lambda^2)}}$$
or 
$$\Kk = \int_{\lambda}^1 \frac{{\rm d} x}{\sqrt{2 (1 - x) (x^2 - \lambda^2)}}.$$ 
In order to calculate $\Kk '$ we integrate up the right edge of $Q_{\kappa}$ as follows: the function ${\rm d}$ defined by ${\rm d} (y) = \dk (\Kk + \ii y)$ satisfies 
$$- ({\rm d} ')^2 = 2 (1 - {\rm d}) ({\rm d}^2 - \lambda^2)$$
and decreases along $[0, \Kk ']$ so that on this interval  
$${\rm d}' = - \sqrt{2 (1 - {\rm d}) (\lambda^2 - {\rm d}^2)}$$
and therefore 
$$\Kk ' = \int_{- \lambda}^{\lambda} \frac{{\rm d} y}{\sqrt{2 (1 - y) (\lambda^2 - y^2)}}.$$ 

\medbreak 

In preparation for the following theorem, specify the complementary acute angles $\alpha$ and $\beta$ by requiring that $\kappa = \cos \alpha$ and $\lambda = \cos \beta$. Further, when $\gamma$ is an acute angle write 
$$I(\gamma) = \int_0^{\gamma} \frac{\cos \frac{1}{2} t}{\sqrt{\cos^2 t - \cos^2 \gamma}} \, {\rm d} t.$$

\bigbreak 

\begin{theorem} \label{IK}
The fundamental periods $2 \Kk$ and $2 \ii \Kk '$ of $\dk$ are given by 
$$\Kk = I (\beta) \; \; {\rm and} \; \; \Kk ' = \sqrt2 \, I(\alpha).$$ 
\end{theorem} 

\begin{proof} 
The following argument is extracted from [4] Theorem 3.3. In the integral formula 
$$\Kk = \int_{\lambda}^1 \frac{{\rm d} x}{\sqrt{2 (1 - x) (x^2 - \lambda^2)}}$$ 
substitute $x = \cos t$ and (as agreed) $\lambda = \cos \beta$: there follows 
$$\Kk = \int_{\beta}^0 \frac{- \sin t }{\sqrt{2 (1 - \cos t)(\cos^2 t - \cos^2 \beta)}}\, {\rm d} t$$
whence by trigonometric dimidiation and cancellation 
$$\Kk = \int_0^{\beta} \frac{\cos \frac{1}{2} t}{\sqrt{\cos^2 t - \cos^2 \beta}} \, {\rm d} t.$$
This establishes the first of the advertised formulae; the second is a little more demanding. In the integral formula 
$$\Kk ' = \int_{- \lambda}^{\lambda} \frac{{\rm d} y}{\sqrt{2 (1 - y) (\lambda^2 - y^2)}}$$ 
substitute first 
$$y = \cos \theta \; \; {\rm where} \; \; \beta \leqslant \theta \leqslant \pi - \beta$$ 
and then 
$$t = \theta - \tfrac{1}{2} \pi \; \; {\rm where} \; \; - \alpha = \beta - \tfrac{1}{2} \pi \leqslant t \leqslant \tfrac{1}{2} \pi - \beta = \alpha.$$ 
The first substitution leads to  
$$\Kk ' = \int_{\beta}^{\pi - \beta} \frac{\cos \frac{1}{2} \theta}{\sqrt{\cos^2 \beta - \cos^2 \theta}} \, {\rm d} \theta$$
essentially as above, whereupon the second substitution gives  
$$\Kk ' = \frac{1}{\sqrt2} \, \int_{- \alpha}^{\alpha} \frac{(\cos \frac{1}{2} t - \sin \frac{1}{2} t)}{\sqrt{\cos^2 t - \cos^2 \alpha}} \, {\rm d} t$$ 
because $\cos^2 \alpha + \cos^2 \beta = 1 = \cos^2 t + \cos^2 \theta$; finally, trigonometric parities justify 
$$\Kk ' = \sqrt2 \, \int_0^{\alpha} \frac{\cos \frac{1}{2} t}{\sqrt{\cos^2 t - \cos^2 \alpha}} \, {\rm d} t.$$
\end{proof} 

\medbreak 

An interesting conclusion to draw here is the symmetric pair of relations 
$$\Kk ' = \sqrt2 \, \Kl \; \; {\rm and} \; \; \Kl ' = \sqrt2 \, \Kk$$ 
involving the complementary moduli $\kappa$ and $\lambda$. These relations have implications for the size and shape of the half-period rectangles $Q_{\kappa}$ and $Q_{\lambda}$ associated to the corresponding elliptic functions $\dk$ and $\dl$. Thus 
$$\Kk \, \Kk ' = \Kl \, \Kl '$$
(so that the half-period rectangles have equal area) and 
$$\frac{\Kk '}{\Kk} \, \frac{\Kl '}{\Kl} = 2.$$ 

\medbreak 

An alternative evaluation of $\Kk$ and $\ii \Kk '$ proceeds by recognizing that the formulae derived leading up to Theorem \ref{IK} express these `half-periods' as elliptic integrals. For this evaluation, we quote the following results from Section 43 of the classic text [2] by Greenhill, with minor notational adjustments. 

\medbreak 

Let the cubic 
$$T = (t - a) (t - b) (t - c)$$ 
have real roots 
$$a > b > c.$$ 
Then 
$$\int_b^a \frac{{\rm d} t}{\sqrt{- T}} = \frac{2}{\sqrt{a - c}} \, K\Big(\frac{a - b}{a - c}\, \Big)$$
and 
$$\int_c^b \frac{{\rm d} t}{\sqrt{T}} = \frac{2}{\sqrt{a - c}} \, K\Big(\frac{b - c}{a - c}\, \Big)$$
where if $0 < k < 1$ then $K(k^2)$ stands for the complete elliptic integral given by 
$$K(k^2) = \int_0^{\frac{1}{2} \pi} \frac{{\rm d} \theta}{\sqrt{1 - k^2 \sin^2 \theta}} = \tfrac{1}{2} \pi F(\tfrac{1}{2} , \tfrac{1}{2} ; 1 ; k^2).$$

\bigbreak 

\begin{theorem} \label{Kell}
The fundamental periods $2 \Kk$ and $2 \ii \Kk '$ of $\dk$ are given by 
$$\Kk = \sqrt{\frac{2}{1 + \lambda}} \, K\Big(\frac{1 - \lambda}{1 + \lambda}\, \Big)$$
and 
$$\Kk ' = \sqrt{\frac{2}{1 + \lambda}} \, K\Big(\frac{2 \lambda}{1 + \lambda}\, \Big).$$
\end{theorem} 

\begin{proof} 
Here, we take the cubic $T$ given by 
$$T = (t - 1) (t - \lambda) (t + \lambda)$$and apply the foregoing results quoted from [2] with $a = 1$, $b = \lambda$ and $c = - \lambda$ so that 
$$a - c = 1 + \lambda, \; \; \frac{a - b}{a - c} = \frac{1 - \lambda}{1 + \lambda} \; \; {\rm and} \; \; \frac{b - c}{a - c} = \frac{2 \lambda}{1 + \lambda}.$$
\end{proof} 

\medbreak 

Otherwise said, 
$$\Kk = \tfrac{1}{2} \, \pi \, \sqrt{\frac{2}{1 + \lambda}} \, F\Big(\frac{1}{2} , \frac{1}{2} ; 1 ; \frac{1 - \lambda}{1 + \lambda}\, \Big)$$
and 
$$\Kk ' = \tfrac{1}{2} \, \pi \, \sqrt{\frac{2}{1 + \lambda}} \, F\Big(\frac{1}{2} , \frac{1}{2} ; 1 ; \frac{2 \lambda}{1 + \lambda}\, \Big).$$

\bigbreak 

The results of Theorem \ref{Kell} can equivalently be approached from a Weierstrassian direction. Let $\dk$ have coperiodic Weierstrass function $\wp$ as in Theorem \ref{dell}. As noted prior to Theorem \ref{values}, the midpoint values of $\wp$ (named according to one of the standard conventions) are 
$$e_1 = \tfrac{1}{6} + \tfrac{1}{2} \lambda, \, e_2 = \tfrac{1}{6} - \tfrac{1}{2} \lambda, \, e_3 = - \tfrac{1}{3}.$$ 
At this stage, we merely quote from Section 51 of [2]: the half-period $\omega = \Kk$ of $\wp$ is given by 
$$\omega = \frac{K (k^2)}{\sqrt{e_1 - e_3}}$$  
where 
$$e_1 - e_3 = \frac{1 + \lambda}{2}$$ 
and 
$$k^2 = \frac{e_2 - e_3}{e_1 - e_3} = \frac{1 - \lambda}{1 + \lambda}$$
whence we recover the previous formula for $\Kk$. Similarly, 
$$\omega ' = \ii \, \frac{K (k'^{ \, 2})}{\sqrt{e_1 - e_3}}$$  
where now 
$$k '{ \,^2} = \frac{e_1 - e_2}{e_1 - e_3} = \frac{2 \lambda}{1 + \lambda}$$
and we recover the previous formula for $\Kk '$. 

\medbreak 

The same results can be approached from a Jacobian elliptic direction. Recall from Theorem \ref{sn} that 
$$\dk (z) = 1 - (1 - \lambda) \, \sn^2 \Big(\sqrt{\frac{1 + \lambda}{2}}\,z \, \Big)$$
where $\sn = \sn (\bullet, k)$ is the Jacobian sine function with modulus $k$ given by 
$$k^2 = \frac{1 - \lambda}{1 + \lambda}\,.$$ 
Now, from $\dk (\Kk) = \lambda$ it follows that $\sn^2 \Big(\sqrt{\frac{1 + \lambda}{2}} \, \Kk \, \Big) = 1$; moreover, if $0 < x < \Kk$ then $1 > \dk (x) > \lambda$ so that $0 < \sn^2 \Big(\sqrt{\frac{1 + \lambda}{2}}x \, \Big) < 1$. It follows easily that $\sqrt{\frac{1 + \lambda}{2}} \, \Kk$ is the least positive $x$ such that $\sn (x) = 1$: that is, 
$$\sqrt{\frac{1 + \lambda}{2}} \, \Kk = K\Big(\frac{1 - \lambda}{1 + \lambda}\, \Big)$$
or again 
$$\Kk = \sqrt{\frac{2}{1 + \lambda}} \, K\Big(\frac{1 - \lambda}{1 + \lambda}\, \Big).$$
Similarly up the imaginary axis: $\dk$ encounters its first (double) pole at $\ii \Kk '$ while $\sn$ encounters its first (simple) pole at 
$$\ii K ' \Big(\frac{1 - \lambda}{1 + \lambda}\, \Big) = \ii K\Big(\frac{2 \lambda}{1 + \lambda}\, \Big)$$
so that again 
$$\Kk ' =  \sqrt{\frac{2}{1 + \lambda}} \, K\Big(\frac{2 \lambda}{1 + \lambda}\, \Big).$$

\bigbreak

\medbreak 

Certain values of the modulus $\kappa$ perhaps deserve special mention. The self-complementary value $\kappa = 1 / \sqrt2$ is plainly singled out: according to the conclusions drawn from Theorem \ref{IK}, the period rectangle of $\dk = \dl$ (and of the coperiodic Weierstrass function) is given by the ratio $\Kk ' / \Kk = \sqrt2$; as expected, the corresponding Jacobian modulus $k$ is $\sqrt2 - 1$ because 
$$k^2 = \frac{1 - \lambda}{1 + \lambda} = \frac{\sqrt2 - 1}{\sqrt2 + 1} = (\sqrt2 - 1)^2.$$
Next, let $\kappa = 2 \sqrt2 / 3$ so that $\lambda = 1 / 3$. In this case, Theorem \ref{dell} shows that the coperiodic Weierstrass function has $g_3 = 0$ so that the period lattice is square: this geometric fact is evident from Theorem \ref{sn} because the corresponding Jacobian modulus $k = 1 / \sqrt2$ is self-complementary; it is also evident from Theorem \ref{Kell} because $ 1 - \lambda = 2 \lambda$. Instead, let $\kappa = 1/3$ so that $\lambda = 2 \sqrt2 / 3$: in this complementary case, the discussion after Theorem \ref{IK} shows that $\Kk ' / \Kk = 2$ and the corresponding Jacobian modulus has the familiar value 
$$k = \sqrt{\frac{1 - \lambda}{1 + \lambda}} = \sqrt{\frac{3 - 2 \sqrt2}{3 + 2 \sqrt2}} = 3 - 2 \sqrt2 = \tan^2 \frac{\pi}{8}.$$

\medbreak

\section*{Hypergeometric identities} 

\medbreak 

Thus far, we have avoided what may be the most natural formula for the fundamental period $2 \Kk$ of $\dk$. We proceed to rectify this omission, after which we extract from this formula and those of the previous section a number of hypergeometric function identities. None of these identities is new; however, our approach to them presents new aspects and opens up the possibility of further identities. 

\medbreak 

\medbreak 

For the following paragraph, we temporarily return to the earlier notation $\dd$ in place of $\dk$ and $K$ in place of $\Kk$. 

\medbreak 

Recall from the close of our `Preliminary Results' that both $\sss$ and $\phi$ extend to functions that are holomorphic in the open band 
$\{ z \in \C : |{\rm Im} \, z| < K ' \}$
where they continue to satisfy the defining relation $\sss = \sin \circ \phi.$ Recall also the identity 
$$\dd^2 + \kappa^2 \sss^2 = 1$$ 
from which we deduce, as we did for $\dd$ and $\sn$ in the discussion following Theorem \ref{Kell}, that $\sss (K) = 1$ and indeed that $\phi (K) = \tfrac{1}{2} \pi$. Lastly, recall that if $u \in \R$ then 
$$u = \int_0^{\phi (u)} F(\tfrac{1}{4}, \tfrac{3}{4}; \tfrac{1}{2} ; \kappa^2 \sin^2 t) \, {\rm d} t.$$ 

\medbreak 

Let us now reinstate the more informative notation $\dk$ and $\Kk$.

\medbreak 

\begin{theorem} \label{F4}
The fundamental half-period $\Kk$ of $\dk$ is given by 
$$\Kk = \tfrac{1}{2} \pi \, F(\tfrac{1}{4}, \tfrac{3}{4} ; 1 ; \kappa^2 ).$$ 
\end{theorem} 

\begin{proof} 
It is a familiar fact that if $a$ and $b$ are arbitrary then 
$$\int_0^{\frac{1}{2} \pi} F(a, b ; \tfrac{1}{2} ; \kappa^2 \sin^2 t) \, {\rm d} t = \tfrac{1}{2} \pi F(a, b ; 1 ; \kappa^2)$$ 
as may be verified by termwise integration of the hypergeometric series. Let $a = \tfrac{1}{4}$ and $b = \tfrac{3}{4}$; then summon the recollections that prefaced this Theorem.  
\end{proof} 

\medbreak 

We are now ready to extract hypergeometric identities. 

\medbreak 

Firstly, direct comparison of Theorem \ref{F4} with the first formula after Theorem \ref{Kell} at once yields the hypergeometric identity 
$$F(\frac{1}{4}, \frac{3}{4} ; 1 ; 1 - \lambda^2 ) = \sqrt{\frac{2}{1 + \lambda}} F\Big(\frac{1}{2} , \frac{1}{2} ; 1 ; \frac{1 - \lambda}{1 + \lambda}\, \Big)$$
when it is recalled that $\kappa^2 = 1 - \lambda^2$. This identity is proved as Theorem 9.2 in [1]. If $\lambda$ is replaced by $(1 - x) / (1 + x)$ - equivalently, if $x$ is replaced by $(1 - \lambda) / (1 + \lambda)$ - then this identity becomes the quadratic transform listed under Corollary 4.2 in [4]. 

\medbreak 

Secondly, direct comparison of Theorem \ref{F4} with the second formula after Theorem \ref{Kell} at once yields the hypergeometric identity 
$$F(\frac{1}{4}, \frac{3}{4} ; 1 ; \lambda^2 ) = \sqrt{\frac{1}{1 + \lambda}} F\Big(\frac{1}{2} , \frac{1}{2} ; 1 ; \frac{2 \lambda}{1 + \lambda}\, \Big)$$
when $\Kk '= \sqrt2 \, \Kl$ is recalled from Theorem \ref{IK}. This identity is recorded on page 260 of Ramanujan's second notebook; it is proved as Theorem 9.1 in [1]. 

\medbreak 

Also recorded on page 260 of Ramanujan's second notebook is the following transformation law. It appears as Theorem 9.4 in [1]; there, it is proved by analytic continuation from small values of the variable. It also appears as Lemma 6.2 in [4]; there, the proof involves theta functions and Landen transformations.  

\medbreak 

\begin{theorem} \label{tr} 
If $0 < x < 1$ then 
$$\sqrt{1 + 3 x} \,  F\Big(\frac{1}{4}, \frac{3}{4} ; 1 ; x^2 \Big) = F\Big(\frac{1}{4}, \frac{3}{4} ; 1 ; 1 - \Big(\frac{1 - x}{1 + 3 x}\Big)^2 \Big).$$
\end{theorem} 

\begin{proof} 
Let $x$ and $y$ be related by the condition 
$$\frac{2 x}{1 + x} = \frac{1 - y}{1 + y}$$ 
and consider the equation 
$$F\Big(\frac{1}{2} , \frac{1}{2} ; 1 ; \frac{2 x}{1 + x}\, \Big) = F\Big(\frac{1}{2} , \frac{1}{2} ; 1 ; \frac{1 - y}{1 + y}\, \Big).$$
On the left side, 
$$F\Big(\frac{1}{2} , \frac{1}{2} ; 1 ; \frac{2 x}{1 + x}\, \Big) = \sqrt{1 + x} F(\frac{1}{4}, \frac{3}{4} ; 1 ; x^2 )$$
by virtue of the second hypergeometric identity displayed just prior to the present Theorem. On the right side, 
$$F\Big(\frac{1}{2} , \frac{1}{2} ; 1 ; \frac{1 - y}{1 + y}\, \Big) = \sqrt{\frac{1 + y}{2}} \, F(\frac{1}{4}, \frac{3}{4} ; 1 ; 1 - y^2 )$$
by virtue of the first hypergeometric identity displayed prior to this Theorem. Finally, the condition relating $x$ and $y$ is equivalent both to   
$$y = \frac{1 - x}{1 + 3 x}$$
and to 
$$\frac{1 + y}{2} = \frac{1 + x}{1 + 3 x}$$  
\medbreak
\noindent
so substitutions and cancellation conclude the proof. 
\end{proof} 

\medbreak 

Incidentally (and perhaps contrary to appearances) the conditions relating $x$ and $y$ in this proof are actually symmetric in $x$ and $y$: they are equivalent to $x + y + 3 x y = 1$. 

\medbreak

\bigbreak

\begin{center} 
{\small R}{\footnotesize EFERENCES}
\end{center} 
\medbreak 

[1] B.C. Berndt, S. Bhargava, and F.G. Garvan, {\it Ramanujan's theories of elliptic functions to alternative bases}, Transactions of the American Mathematical Society {\bf 347} (1995) 4163-4244. 

\medbreak 

[2] A.G. Greenhill, {\it The Applications of Elliptic Functions}, Macmillan and Company (1892); Dover Publications (1959). 

\medbreak 

[3] P.L. Robinson, {\it Elliptic functions from $F(\tfrac{1}{4}, \tfrac{3}{4}; \tfrac{1}{2} ; \bullet)$}, arXiv 1908.01687 (2019). 

\medbreak 

[4] Li-Chien Shen, {\it On a theory of elliptic functions based on the incomplete integral of the hypergeometric function $_2 F_1 (\frac{1}{4}, \frac{3}{4} ; \frac{1}{2} ; z)$}, Ramanujan Journal {\bf 34} (2014) 209-225. 

\medbreak

\end{document}